\documentclass{amsproc}

\usepackage{mathrsfs}
\usepackage{amsmath}
\usepackage{amssymb}
\usepackage{amsthm}
\usepackage{mathrsfs}
\usepackage{amscd}

\newtheorem{theorem}{Theorem}[section]
\newtheorem{lemma}[theorem]{Lemma}

\newtheorem{corollary}[theorem]{Corollary}

\theoremstyle{definition}

\newtheorem{example}[theorem]{Example}

\theoremstyle{remark}
\newtheorem{remark}[theorem]{Remark}

\numberwithin{equation}{section}

\begin{document}

\title[Ergodic value distribution]{An ergodic value distribution of certain meromorphic functions}


%    author one information
\author[J. Lee]{Junghun Lee}
\address{Graduate School of Mathematics, Nagoya University, Furo-cho, Chikusa-ku,
Nagoya 464-8602, Japan}
%\curraddr{}
\email{m12003v@math.nagoya-u.ac.jp}
%\thanks{The authors would like to thank to Professor Kohji Matsumoto and the referees.}

%    author two information
\author[A. I. Suriajaya]{Ade Irma Suriajaya}
\address{Graduate School of Mathematics, Nagoya University, Furo-cho, Chikusa-ku,
Nagoya 464-8602, Japan}
%\curraddr{}
\email{m12026a@math.nagoya-u.ac.jp}
\thanks{The authors are supported by JSPS KAKENHI Grant Number 16J01139 and 15J02325, respectively.}

\subjclass[2010]{Primary 11M06, Secondary 37A05 37A30}
%    The 2010 edition of the Mathematics Subject Classification is
%    now available.  If you are citing a classification from the
%    new scheme, use the following input coding instead.

%\date{25 July 2015}
\keywords{Birkhoff's ergodic theorem, zeta function, $L$-function, derivative}

%%% %%% %%% %%%%% %%%%% %%% %%% %%%
%%% %%% %%% %%%%% %%%%% %%% %%% %%%

\begin{abstract}
We calculate a certain mean-value of meromorphic functions by using specific ergodic transformations, which we call affine Boolean transformations.
We use Birkhoff's ergodic theorem to transform the mean-value into a computable integral which allows us to completely determine the mean-value of this ergodic type.
As examples, we introduce some applications to zeta functions and $L$-functions.
We also prove an equivalence of the Lindel\"{o}f hypothesis of the Riemann zeta function in terms of its certain ergodic value distribution associated with affine Boolean transformations.
\end{abstract}

\maketitle

%  Text of article.

%
\section{Introduction}

In \cite{LW09}, M. Lifshitz and M. Weber investigated the value distribution of the {\it Riemann zeta function} $\zeta(s)$ by using the Cauchy random walk.
They proved that almost surely
$$
\lim_{N \rightarrow \infty} {1 \over N} \sum_{n = 1}^{N} \zeta\left({ 1 \over 2} + i S_n \right) = 1 + o\left( \frac{(\log{N})^b}{N^{1/2}} \right)
$$
holds for any $b > 2$ where $\{ S_n \}_{n = 1}^{\infty}$ is the Cauchy random walk.
This result implies that most of the values of $\zeta(s)$ on the critical line are quite small.
Analogous to \cite{LW09}, T. Srichan investigated the value distributions of Dirichlet $L$-functions and Hurwitz zeta functions by using the Cauchy random walk in \cite{Sri15}.

The first approach to investigate the ergodic value distribution of $\zeta(s)$ was done by J. Steuding.
In \cite{ste}, he studied the ergodic value distribution of $\zeta(s)$ on vertical lines under the Boolean transformation.

We are interested in studying the ergodic value distribution of a larger class of meromorphic functions which includes but is not limited to the Selberg class (of $\zeta$-functions and $L$-functions) and their derivatives, on vertical lines under more general Boolean transformations, which we shall call {\it affine Boolean transformation} $T_{\alpha,\beta}: \mathbb{R} \rightarrow \mathbb{R}$ given by 
\begin{equation} \label{eq:affboot}
\begin{aligned}
T_{\alpha, \beta}(x) : = 
\begin{cases}
\displaystyle {\alpha \over 2} \left( \frac{x + \beta}{\alpha} - {\alpha \over x - \beta } \right), &\quad x \neq \beta; \\
\displaystyle \beta, &\quad x = \beta
\end{cases}
\end{aligned}
\end{equation}
for an $\alpha > 0$ and a $\beta \in \mathbb{R}$.
Below is our main theorem.
For a given $c \in \mathbb{R}$, we shall denote by $\mathbb{H}_c$ and $\mathbb{L}_c$ the half-plane $\{ z \in \mathbb{C} \mid \operatorname{Re}(z) > c \}$ and the line $\{ z \in \mathbb{C} \mid \operatorname{Re}(z) = c \}$.

% % %
\begin{theorem} \label{birk}
Let $f$ be a meromorphic function on $\mathbb{H}_c$ satisfying the following conditions.
\begin{enumerate}
\item There exists an $M > 0$ and a $c' > c$ such that for any $t \in \mathbb{R}$, we have
$$ | f ( \{ \sigma + i t \mid \sigma > c' \}) | \leq M. $$
\item There exists a non-increasing continuous function $\nu:(c,\infty)\rightarrow\mathbb{R}$ such that if $\sigma$ is sufficiently near $c$ then $\nu(\sigma)\leq1+c-\sigma$,
and that for any small $\epsilon>0$, $f(\sigma+it) \ll_{f,\epsilon} |t|^{\nu(\sigma)+\epsilon}$ as $|t|\rightarrow\infty$.
\item $f$ has at most one pole of order $m$ in $\mathbb{H}_c$ at $s=s_0=\sigma_0+it_0$, that is, we can write its Laurent expansion near $s=s_0$ as
\begin{equation} \label{eq:laurent}
\frac{a_{-m}}{(s-s_0)^m} + \frac{a_{-(m-1)}}{(s-s_0)^{m-1}} + \cdots + \frac{a_{-1}}{s-s_0} + a_0 + \sum_{n=1}^\infty a_n(s-s_0)^n
\end{equation}
for $m\geq0$, where we set $m=0$ if $f$ has no pole in $\mathbb{H}_c$.
\end{enumerate}
Then for any $s \in \mathbb{H}_c \backslash \mathbb{L}_{\sigma_0}$, we have
\begin{equation} \label{eq:birk}
\lim_{N \rightarrow\infty} \frac{1}{N} \sum_{n = 0}^{N - 1} f\left(s+iT_{\alpha,\beta}^nx\right)
= \frac{\alpha}{\pi} \int_{\mathbb{R}} \frac{f(s+i\tau)}{\alpha^2 +(\tau-\beta)^2} d\tau
\end{equation}
for almost all $x \in \mathbb{R}$.

We denote the right-hand side of the above formula by $l_{\alpha,\beta}(s)$.
%\begin{itemize}
%
%\item
If $f$ has no pole in $\mathbb{H}_c$,
\begin{equation} \label{eq:1}
l_{\alpha,\beta}(s) = f(s+\alpha+i\beta)
\end{equation}
for all $s \in \mathbb{H}_c$.
%
%\item
If $f$ has a pole at $s=s_0=\sigma_0+it_0$,
\begin{equation} \label{eq:2}
\begin{aligned}
l_{\alpha,\beta}(s) =
\begin{cases}
f(s+\alpha+i\beta) + B_m(s_0), &\,\, c<\operatorname{Re}(s)<\sigma_0, s\neq s_0-\alpha-i\beta;\\
\displaystyle \sum_{n=0}^m \frac{a_{-n}}{(-2\alpha)^n}, &\,\, c<\operatorname{Re}(s)<\sigma_0, s=s_0-\alpha-i\beta;\\
f(s+\alpha+i\beta), &\,\, \operatorname{Re}(s)>\sigma_0;
\end{cases}
\end{aligned}
\end{equation}
where
$$
B_m(s_0) = \sum_{n=1}^m \frac{a_{-n}}{i^n \left( \beta+i\alpha-i(s-s_0) \right)^n}
- \sum_{n=1}^m \frac{a_{-n}}{i^n \left( \beta-i\alpha-i(s-s_0) \right)^n}.
$$
Moreover when $m=1$, we can extend the result in \eqref{eq:birk} to the line $\mathbb{L}_{\sigma_0}$ by setting
\begin{equation} \label{eq:3}
l_{\alpha, \beta}(\sigma_0 + i t) = 
f(\sigma_0+\alpha+i(t+\beta)) - \frac{a_{-1}\alpha}{\alpha^2 + (t_0-t-\beta)^2}
\end{equation}
for any $t \in \mathbb{R}$.
%
%\end{itemize}
\end{theorem}
% % %

% % %
%\begin{remark}
%Even though the calculation is somewhat cumbersome, by the same method, we can also extend these results to the case when $f$ has finitely many poles.
%When considering this case, writing the Laurent expansion of $f$ near each pole $s=s_j$ in the form
%$$
%\frac{a_{j,-m_j}}{(s-s_j)^{m_j}} + \frac{a_{j,-(m_j-1)}}{(s-s_j)^{m_j-1}} + \cdots + \frac{a_{j,-1}}{s-s_j} + a_{j,0} + \sum_{n=1}^\infty a_{j,n}(s-s_j)^n,
%$$
%we need to omit the lines $\mathbb{L}_{\sigma_{j(1)}}, \cdots, \mathbb{L}_{\sigma_{j(r)}}$ where $m_{j(1)}, \cdots, m_{j(r)}>1$.
%The reason can be seen in Case 3 of our proof.
%
%In this paper, we omit the proof of these cases.
%\end{remark}
% % %

%
In the next section, we first give a few examples as applications of our main theorem, Theorem \ref{birk}, to the Riemann zeta function, Dirichlet $L$-functions, Dedekind zeta functions, Hurwitz zeta functions, and their derivatives.
We will briefly review some basics of ergodic theory and see an ergodic property of affine Boolean transformations in Section \ref{sec_aff}.
In Section \ref{sec_proof}, we will complete the proof of Theorem \ref{birk}.
% and Corollary \ref{dir_ser}.

%
\section{Some applications to zeta functions and $L$-functions}

In the following examples, we write $f^{(0)}$ to express $f$ itself and we define
$$
A_k(s) := \frac{(-1)^k k!}{i^{k+1}} \left( \frac{1}{\left( \beta+i\alpha-i(s-1) \right)^{k+1}}
- \frac{1}{\left( \beta-i\alpha-i(s-1) \right)^{k+1}} \right)
$$
for any non-negative integer $k$.

% %1% %
\begin{example}[The Riemann zeta function] \label{birk1}
For any $k \geq 0$ and $s \in \mathbb{H}_{- 1 / 2} \backslash \mathbb{L}_1$, we have
$$
\lim_{N \rightarrow\infty} \frac{1}{N} \sum_{n = 0}^{N - 1} \zeta^{(k)}\left(s+iT_{\alpha,\beta}^nx\right)
= \frac{\alpha}{\pi} \int_{\mathbb{R}} \frac{\zeta^{(k)}(s+i\tau)}{\alpha^2 +(\tau-\beta)^2} d\tau
$$
for almost all $x \in \mathbb{R}$.

Denoting the right-hand side of the above formula by $l_{\alpha,\beta}^{(k)}(s)$, we have
\begin{align*}
l_{\alpha,\beta}^{(k)}(s) = 
\begin{cases}
\zeta^{(k)}(s+\alpha+i\beta) + A_k(s), &\,\, -1/2 < \operatorname{Re}(s)<1, s\neq 1-\alpha-i\beta; \\
\displaystyle (-1)^k\gamma_k - \frac{k!}{(2\alpha)^{k+1}}, &\,\, -1/2 < \operatorname{Re}(s)<1, s = 1-\alpha-i\beta; \\
\zeta^{(k)}(s+\alpha+i\beta), &\,\, \operatorname{Re}(s)>1;
\end{cases}
\end{align*}
where
$$
\gamma_k := \lim_{N\rightarrow\infty} \left( \sum_{n=1}^N \frac{\log^k{n}}{n} - \frac{\log^{k+1}{N}}{k+1} \right).
$$
If $k=0$, we can extend the result to the line $\mathbb{L}_1$ by setting
$$
l_{\alpha,\beta}^{(0)}(1+it) = \zeta^{(0)}(1+\alpha+i(t+\beta)) - \frac{\alpha}{\alpha^2 + (t+\beta)^2}.
$$
\end{example}
% %1% %

Remark that Steuding showed Example \ref{birk1} when $k=0$, $\alpha=1$, and $\beta=0$ thus Example \ref{birk1} is a generalization of \cite[Theorem 1.1]{ste}.

\begin{proof}[Proof of Example \ref{birk1}]
We first note that for any $k\geq0$, $\zeta^{(k)}(s)$ has an absolute convergent Dirichlet series expression when $\operatorname{Re}(s)>1$.
Thus condition (1) of Theorem \ref{birk} is satisfied for any $c'>1$.
From the Laurent expansion of $\zeta(s)$ near its pole $s=1$ (see \cite[Theorem]{bri}), we can deduce the Laurent expansion of $\zeta^{(k)}(s)$ for any $k\geq0$ near $s=1$:
$$
\zeta^{(k)}(s) = \frac{(-1)^k k!}{(s-1)^{k+1}} + (-1)^k\gamma_k + \sum_{n=k}^\infty \frac{(-1)^{n+1}\gamma_{n+1}}{(n-k+1)!} (s-1)^{n-k+1}.
$$
Thus for $k\geq0$, $\zeta^{(k)}(s)$ has a pole of order $k+1$ at $s = 1$.
Moreover, we can show by using \cite[pp. 95--96]{tit2} that
\begin{equation} \label{eq:t-bound}
\zeta^{(k)}(\sigma+it) \ll_{k,\epsilon} |t|^{\mu(\sigma)+\epsilon}
\end{equation}
holds with
\begin{align*}
\mu(\sigma) \leq
\begin{cases}
0 &\text{ if } \sigma>1; \\
(1-\sigma)/2 &\text{ if } 0\leq\sigma\leq1; \\
1/2-\sigma &\text{ if } \sigma<0;
\end{cases}
\end{align*}
for any $k \geq 0$.
Therefore we can apply Theorem \ref{birk} with $c=-1/2$, $s_0=1$, and $m=k+1$ to $\zeta^{(k)}(s)$.
\end{proof}

We can also show that this ergodic mean-value is related to the Lindel\"{o}f hypothesis.
We first show that the Lindel\"{o}f hypothesis can be rewritten in terms of $\zeta^{(k)}(s)$.
% % %
\begin{theorem} \label{lindeloef1}
Let $k\in\mathbb{N}$.
The Lindel\"{o}f hypothesis: For any $\epsilon>0$,
$$
\zeta\left(\frac{1}{2}+it\right) \ll_\epsilon |t|^\epsilon \quad \text{ as } |t|\rightarrow\infty
$$
holds if only if, for any $\epsilon>0$,
$$
\zeta^{(k)}\left(\frac{1}{2}+it\right) \ll_{k,\epsilon} |t|^\epsilon \quad \text{ as } |t|\rightarrow\infty
$$
holds.
\end{theorem}
The above theorem implies that we can restate the Lindel\"{o}f hypothesis as:
$$
\text{For any } \epsilon>0,\,\,
\zeta^{(k)}\left(\frac{1}{2}+it\right) \ll_{k,\epsilon} |t|^\epsilon \quad \text{ as } |t|\rightarrow\infty
$$
for any non-negative integer $k$.

\begin{proof}[Proof of Theorem \ref{lindeloef1}]
Suppose that the Lindel\"{o}f hypothesis is true.
Thus by using the functional equation for $\zeta(s)$,
$$ \zeta(\sigma+it) \ll_\epsilon |t|^{\mu(\sigma)+\epsilon/2} $$
holds for any $\epsilon>0$ with
$$
\mu(\sigma) \leq
\begin{cases}
0 &\text{ if } \sigma\geq1/2; \\
1/2-\sigma &\text{ if } \sigma<1/2.
\end{cases}
$$
Then by Cauchy's integral theorem, for any $k\in\mathbb{N}$ we have
$$
\zeta^{(k)}\left(\frac{1}{2}+it\right)
= \frac{k!}{2\pi i} \int_{\gamma_r} \frac{\zeta(z)}{(z-1/2-it)^{k+1}} dz,
$$
where $\gamma_r := \{ z\in\mathbb{C} \mid |z-1/2-it|=r \}$.
Taking $r=\epsilon/2$,
\begin{align*}
\left| \zeta^{(k)}\left(\frac{1}{2}+it\right) \right|
&\ll_k \int_{\gamma_r} \frac{|\zeta(z)|}{|z-1/2-it|^{k+1}} |dz|
%\ll_{k,\epsilon} |t+\epsilon/2|^{\mu(1/2-\epsilon\/2)+\epsilon/2} \frac{\epsilon/2}{(\epsilon/2)^{k+1}}
\ll_{k,\epsilon} |t+\epsilon/2|^{\mu(1/2-\epsilon/2)+\epsilon/2} \\
&\ll_\epsilon |t|^{\mu(1/2-\epsilon/2)+\epsilon/2}
\leq |t|^{1/2-(1/2-\epsilon/2)+\epsilon/2} = |t|^\epsilon.
\end{align*}

Now suppose that for some $k\in\mathbb{N}$,
$$ \zeta^{(k)}\left(\frac{1}{2}+it\right) \ll_{k,\epsilon} |t|^\epsilon $$
holds for any $\epsilon>0$.
Then $ \zeta^{(k)}(\sigma+it) \ll_{k,\epsilon} |t|^{\mu(\sigma)+\epsilon} $ for $\sigma\geq1/2$.
Note that
$$
\left| \zeta^{(k-1)}\left(\frac{1}{2}+it\right) \right|
\leq \left| \zeta^{(k-1)}(3+it) \right| + \int_{1/2+it}^{3+it} \left| \zeta^{(k)}(z) \right| |dz|
\ll_{k,\epsilon} |t|^\epsilon.
$$
This implies
$$ \left| \zeta\left(\frac{1}{2}+it\right) \right| \ll_\epsilon |t|^\epsilon. $$
\end{proof}

We can then reformulate the Lindel\"{o}f hypothesis in terms of ergodic value distribution of $\zeta^{(k)}(s)$ on vertical lines under affine Boolean transformations as follows:
% % %
\begin{theorem} \label{lindeloef2}
Let $k$ be a non-negative integer.
The Lindel\"{o}f hypothesis is true if and only if, there exist $\alpha>0$, $\beta\in\mathbb{R}$ such that for any $l\in\mathbb{N}$,
\begin{equation} \label{eq:lind.erg}
\lim_{N \rightarrow\infty} \frac{1}{N} \sum_{n = 0}^{N - 1} \left| \zeta^{(k)}\left(1/2+iT_{\alpha,\beta}^nx\right) \right|^{2l}
%= \frac{\alpha}{\pi} \int_{\mathbb{R}} \frac{\left| \zeta^{(k)}(1/2+i\tau) \right|^{2l}}{\alpha^2 +(\tau-\beta)^2} d\tau
\end{equation}
exists for almost all $x \in \mathbb{R}$.
\end{theorem}

\begin{proof}[Proof of Theorem \ref{lindeloef2}]
From Theorem \ref{lindeloef1}, we can restate the Lindel\"{o}f hypothesis as
\begin{equation} \label{eq:lind.k}
\zeta^{(k)}\left(\frac{1}{2}+it\right) \ll_{k,\epsilon} |t|^\epsilon \quad \text{ as } |t|\rightarrow\infty
\end{equation}
for any non-negative integer $k$.
We then show that the hypothesis in the form \eqref{eq:lind.k} is equivalent to the existence of the limit in \eqref{eq:lind.erg}.

Replacing the function $\zeta(s)$ by $\zeta^{(k)}(s)$ in the proof of Theorem 4.1 in \cite{ste}, we can easily show the necessary condition for the Lindel\"{o}f hypothesis (in the form \eqref{eq:lind.k}).

To show the sufficient condition for the Lindel\"{o}f hypothesis, we note that
%$$
%\zeta^{(k)}(s) = (-1)^k s\int_1^\infty \frac{[x]-x+1/2}{x^{s+1}}(\log{x})^k dx
%+ (-1)^{k-1} k\int_1^\infty \frac{[x]-x+1/2}{x^{s+1}}(\log{x})^{k-1} dx
%+ \frac{(-1)^k k!}{(s-1)^{k+1}}
%$$
$$
\zeta^{(k)}(s)
= (-1)^{k-1} \int_1^\infty \frac{[x]-x+1/2}{x^{s+1}}(\log{x})^{k-1} \left( -s\log{x} + k \right) dx
+ \frac{(-1)^k k!}{(s-1)^{k+1}}
$$
so that $|\zeta^{(k+1)}(1/2+it)| < C_k |t|$ holds for any $|t|\geq1$ for some $C_k>0$ which may depend only on $k$.
Further, for $\tau\geq1$,
$$
\frac{1}{\alpha^2 +(\tau-\beta)^2}
= \frac{1}{\tau^2 (1+ (\alpha/\tau)^2 + 2|\beta|/\tau + (\beta/\tau)^2)}
\geq C_{\alpha,\beta} \frac{1}{\tau^2} \geq C_{\alpha,\beta} \frac{1}{1+ \tau^2}
$$
for some $C_{\alpha,\beta}>0$ that depends only on $\alpha$ and $\beta$.
Then again we can replace the function $\zeta(s)$ by $\zeta^{(k)}(s)$ in the proof of Theorem 4.1 in \cite{ste} to obtain the sufficient condition for the Lindel\"{o}f hypothesis (in the form \eqref{eq:lind.k}).
This completes our proof of Theorem \ref{lindeloef2}.
\end{proof}

% %2% %
\begin{example}[Dirichlet $L$-functions] \label{birk2}
Let $L(s, \chi)$ be the {\it Dirichlet $L$-function} associated with Dirichlet character $\chi$.
\begin{enumerate}
\item[(i)] If $\chi$ is non-principal, for any $s \in \mathbb{H}_{- 1 / 2}$, we have
\begin{align*}
\lim_{N \rightarrow\infty} \frac{1}{N} \sum_{n = 0}^{N - 1} L^{(k)}\left(s+iT_{\alpha,\beta}^nx,\chi\right)
&= \frac{\alpha}{\pi} \int_{\mathbb{R}} \frac{L^{(k)}(s+i\tau,\chi)}{\alpha^2 +(\tau-\beta)^2} d\tau \\
&= L^{(k)}(s+\alpha+i\beta,\chi)
\end{align*}
for almost all $x \in \mathbb{R}$.
\item[(ii)] If $\chi=\chi_0$ is principal, for any $s \in \mathbb{H}_{- 1 / 2} \backslash \mathbb{L}_{1}$, we have
$$
\lim_{N \rightarrow\infty} \frac{1}{N} \sum_{n = 0}^{N - 1} L^{(k)}\left(s+iT_{\alpha,\beta}^nx,\chi_0\right)
= \frac{\alpha}{\pi} \int_{\mathbb{R}} \frac{L^{(k)}(s+i\tau,\chi_0)}{\alpha^2 +(\tau-\beta)^2} d\tau
$$
for almost all $x \in \mathbb{R}$.
Denoting the right-hand side of the above formula by $l_{\alpha,\beta}^{(k)}(s,\chi_0)$, we have
\begin{align*}
\hskip15mm
l_{\alpha,\beta}^{(k)}(s,\chi_0) = 
\begin{cases}
L^{(k)}(s+\alpha+i\beta,\chi_0) + \gamma_{-1}(\chi_0) A_k(s), \\
\quad\quad\quad\quad\quad\quad\quad\quad\quad\quad
-1/2<\operatorname{Re}(s)<1, s\neq 1-\alpha-i\beta; \\
\displaystyle \gamma_k(\chi_0) - \frac{k!\gamma_{-1}(\chi_0)}{(2\alpha)^{k+1}}, \\
\quad\quad\quad\quad\quad\quad\quad\quad\quad\quad
-1/2<\operatorname{Re}(s)<1, s = 1-\alpha-i\beta; \\
L^{(k)}(s+\alpha+i\beta,\chi_0), \\
\quad\quad\quad\quad\quad\quad\quad\quad\quad\quad
\operatorname{Re}(s)>1;
\end{cases}
\end{align*}
where $\gamma_{-1}(\chi_0)$, $\gamma_k(\chi_0)$'s are constants that depend only on $\chi_0$.
They are coefficients of the Laurent expansion of $L^{(k)}(s,\chi_0)$ near $s=1$.
If $k=0$, we can also show the result on $\mathbb{L}_1$ by setting
$$
l_{\alpha,\beta}^{(0)}(1+it,\chi_0) = L^{(0)}(1+\alpha+i(t+\beta),\chi_0) - \frac{\alpha \gamma_{-1}(\chi_0)}{\alpha^2 + (t+\beta)^2}.
$$
\end{enumerate}
\end{example}

\begin{proof}[Proof of Example \ref{birk2}]
As in the proof of Example \ref{birk1}, for any non-negative integer $k$, $L^{(k)}(s,\chi)$ has an absolute convergent Dirichlet series expression when $\operatorname{Re}(s)>1$.
Referring to \cite[Lemma 2]{red}, we know that $L^{(k)}(s,\chi)$ also satisfies an inequality similar to \eqref{eq:t-bound}.

If $\chi$ is non-principal, $L^{(k)}(s,\chi)$ is entire for all $k\geq0$.
Thus $L^{(k)}(s,\chi)$ satisfies \eqref{eq:1} of Theorem \ref{birk} for all $s \in \mathbb{H}_{-1/2}$.

Otherwise (that is, when $\chi=\chi_0$), $L^{(k)}(s,\chi_0)$ ($k\geq1$) has a pole of order $k+1$ at $s = 1$.
Hence we can also apply Theorem \ref{birk} with $c=-1/2$, $s_0=1$, and $m=k+1$ to $L^{(k)}(s,\chi_0)$ with the Laurent coefficients as discussed in \cite[Theorem 2]{kan}.
\end{proof}
% %2% %

% %3% %
\begin{example}[Dedekind $\zeta$-functions] \label{birk3}
Let $\zeta_\mathbb{K}(s)$ be the {\it Dedekind $\zeta$-function} of a number field $\mathbb{K}$ over $\mathbb{Q}$ of degree $d_\mathbb{K}$.
Then for any $k \geq 0$ and $s \in \mathbb{H}_{1/2 - 1/d_\mathbb{K}} \backslash \mathbb{L}_{1}$, we have
$$
\lim_{N \rightarrow\infty} \frac{1}{N} \sum_{n = 0}^{N - 1} \zeta_\mathbb{K}^{(k)}\left(s+iT_{\alpha,\beta}^nx\right)
= \frac{\alpha}{\pi} \int_{\mathbb{R}} \frac{\zeta_\mathbb{K}^{(k)}(s+i\tau)}{\alpha^2 +(\tau-\beta)^2} d\tau
$$
for almost all $x \in \mathbb{R}$.

Denoting the right-hand side of the above formula by ${l_\mathbb{K}}_{\alpha,\beta}^{(k)}(s)$, we have
\begin{align*}
{l_\mathbb{K}}_{\alpha,\beta}^{(k)}(s) = 
\begin{cases}
\zeta_\mathbb{K}^{(k)}(s+\alpha+i\beta) + \gamma_{-1}(\mathbb{K}) A_k(s), \\
\quad\quad\quad\quad\quad\quad\quad\quad\quad\quad
1/2 - 1/d_\mathbb{K} < \operatorname{Re}(s) < 1, s \neq 1-\alpha-i\beta; \\
\displaystyle k!\gamma_k(\mathbb{K}) - \frac{k!\gamma_{-1}(\mathbb{K})}{(2\alpha)^{k+1}}, \\
\quad\quad\quad\quad\quad\quad\quad\quad\quad\quad
1/2 - 1/d_\mathbb{K} < \operatorname{Re}(s) < 1, s = 1-\alpha-i\beta; \\
\zeta_\mathbb{K}^{(k)}(s+\alpha+i\beta), \\
\quad\quad\quad\quad\quad\quad\quad\quad\quad\quad
\operatorname{Re}(s)>1;
\end{cases}
\end{align*}
where $\gamma_{-1}(\mathbb{K})$, $\gamma_k(\mathbb{K})$'s are constants that depend only on $\mathbb{K}$.
They are coefficients of the Laurent expansion of $\zeta_\mathbb{K}^{(k)}(s)$ near $s=1$.
If $k=0$, we can also show the result on $\mathbb{L}_1$ by setting
$$
{l_\mathbb{K}}_{\alpha,\beta}^{(0)}(1+it) = \zeta_\mathbb{K}^{(0)}(1+\alpha+i(t+\beta)) - \frac{\alpha\gamma_{-1}(\mathbb{K})}{\alpha^2 + (t+\beta)^2}.
$$
\end{example}

\begin{proof}[Proof of Example \ref{birk3}]
We refer to \cite[Theorem 2]{ste1} for the bound of the form \eqref{eq:t-bound} and to \cite[pp. 496--497]{has} for the Laurent coefficients of $\zeta_\mathbb{K}^{(k)}(s)$ near its pole at $s=1$.
The rest of the proof proceeds as in the proof of Example \ref{birk1} with $c=1/2-1/d_\mathbb{K}$, $s_0=1$, and $m=k+1$.
\end{proof}
% %3% %

% % %
\begin{remark}
We can also show results analogous to Theorems \ref{lindeloef1} and \ref{lindeloef2} for Dirichlet $L$-functions associated with primitive Dirichlet characters and Dedekind zeta functions, if we formulate the \emph{extended Lindel\"{o}f hypothesis} as:
$$
\text{For any } \epsilon>0,\,\,
f\left(\frac{1}{2}+it\right) \ll_{f,\epsilon} |t|^\epsilon \quad \text{ as } |t|\rightarrow\infty
$$
for these functions ($f$ is any of these $\zeta$-functions and $L$-functions).
We do not discuss this further but we remark that we can show these analogous results by using methods similar to the methods used in proving Theorems \ref{lindeloef1} and \ref{lindeloef2}.
\end{remark}

% %4% %
\begin{example}[Hurwitz zeta functions] \label{birk4}
For non-negative integer $k$, $0<\mathfrak{a}\leq1$, and any $s$ satisfying $\operatorname{Re}(s)>-1/2$ and $\operatorname{Re}(s)\neq1$, we have
$$
\lim_{N \rightarrow\infty} \frac{1}{N} \sum_{n = 0}^{N - 1} \zeta^{(k)}(s+iT_{\alpha,\beta}^nx,\mathfrak{a})
= \frac{\alpha}{\pi} \int_{\mathbb{R}} \frac{\zeta^{(k)}(s+i\tau,\mathfrak{a})}{\alpha^2 +(\tau-\beta)^2} d\tau
$$
for almost all $x$ in $\mathbb{R}$.

Denoting the right-hand side of the above formula by $l_{\alpha,\beta}^{(k)}(s,\mathfrak{a})$, we have
\begin{align*}
l_{\alpha,\beta}^{(k)}(s,\mathfrak{a}) = 
\begin{cases}
\zeta^{(k)}(s+\alpha+i\beta,\mathfrak{a}) + A_k(s), \\
\quad\quad\quad\quad\quad\quad\quad\quad\quad\quad
-1/2 < \operatorname{Re}(s) < 1, s \neq 1-\alpha-i\beta; \\
\displaystyle k!\gamma_k(\mathfrak{a}) - \frac{k!}{(2\alpha)^{k+1}}, \\
\quad\quad\quad\quad\quad\quad\quad\quad\quad\quad
-1/2 < \operatorname{Re}(s) < 1, s = 1-\alpha-i\beta; \\
\zeta^{(k)}(s+\alpha+i\beta,\mathfrak{a}), \\
\quad\quad\quad\quad\quad\quad\quad\quad\quad\quad
\operatorname{Re}(s)>1;
\end{cases}
\end{align*}
where
$$
\gamma_k(\mathfrak{a}) := \frac{(-1)^k}{k!} \lim_{N\rightarrow\infty} \left( \sum_{n=0}^N \frac{\log^k{(n+\mathfrak{a})}}{n+\mathfrak{a}} - \frac{\log^{k+1}{(N+\mathfrak{a})}}{k+1} \right)
$$
is a coefficient of the Laurent expansion of $\zeta_\mathbb{K}^{(k)}(s)$ near $s=1$.
If $k=0$, we can also show the result on $\mathbb{L}_1$ by setting
$$
l_{\alpha,\beta}^{(0)}(1+it,\mathfrak{a}) = \zeta^{(0)}(1+\alpha+i(t+\beta),\mathfrak{a}) - \frac{\alpha}{\alpha^2 + (t+\beta)^2}.
$$
\end{example}

\begin{proof}[Proof of Example \ref{birk4}]
The proof also follows that of Example \ref{birk1} where we put $c=-1/2$, $s_0=1$, and $m=k+1$.
Here, we refer to \cite[Lemma 2]{red} for the bound of the form \eqref{eq:t-bound} and to \cite[Theorem 1]{ber} for the Laurent coefficients of $\zeta^{(k)}(s,\mathfrak{a})$ near its pole at $s=1$.
\end{proof}
% %4% %

%%%%%%%%%%%%%%%%LEE%%%%%%%%%%%%%%%%%%%%%%
%%%%%%%%%%%%%%%%LEE%%%%%%%%%%%%%%%%%%%%%%

%%%%%%%%%%%%%%
%%%%%%%%%%%%%%
\section{Affine Boolean transformations} \label{sec_aff}
%%%%%%%%%%%%%%
%%%%%%%%%%%%%%

In this section, we will show the ergodicity of $T_{\alpha, \beta}$ defined in \eqref{eq:affboot} with respect to a proper measure.
To state our main theorem, let us recall some basic notation.
We denote by $\mathcal{B}$ and $\nu$ the {\it Borel $\sigma$-algebra} on $\mathbb{R}$ and the {\it Lebesgue measure} on $\mathcal{B}$.
For a given $\alpha > 0, \beta \in \mathbb{R}$, let us define the function $\mu_{\alpha, \beta}$ by
$$
\mu_{\alpha, \beta}(A) := {\alpha \over \pi} \int_{A} {d \tau \over \alpha^2 + (\tau - \beta)^2 }
$$
for any $A \in \mathcal{B}$.
One can easily check that $\mu_{\alpha, \beta}$ is a probability on $\mathcal{B}$ and
\begin{equation}\label{eq: zero}
\mu_{\alpha, \beta}(A) 
= {\alpha \over \pi} \int_{A} {d \tau \over \alpha^2 + (\tau - \beta)^2 } 
\leq \int_{A} { d \tau \over \alpha \pi } 
= {1 \over \alpha \pi} \nu(A)
\end{equation}
for any $A \in \mathcal{B}$.
In particular, this implies that $\mu_{\alpha, \beta}(A) = 0$ if $\nu(A) = 0$.

%%%%
\begin{theorem} \label{001}
For given $\alpha > 0 , \beta \in \mathbb{R}$, $T_{\alpha, \beta}: \mathbb{R} \rightarrow \mathbb{R}$ is measure preserving with respect to $\mu_{\alpha, \beta}$, that is, for any $A \in \mathcal{B}$, we have 
$$
\mu_{\alpha, \beta}(T_{\alpha, \beta}^{-1}(A)) = \mu_{\alpha, \beta}(A).
$$
Moreover, it is ergodic, that is, if $T_{\alpha, \beta}^{-1}(A) = A$, then either $\mu_{\alpha, \beta}(A)$ or $\mu_{\alpha, \beta}(X \backslash A)$ is $0$.
\end{theorem}
%%%%

Applying Birkhoff's ergodic theorem, we have an ergodic mean-value of an integrable function.
Let us denote by $T_{\alpha, \beta}^n$ the {\it $n$-th iteration} of $T_{\alpha, \beta}$, that is, 
$$
T_{\alpha, \beta}^n := \underbrace{T_{\alpha, \beta} \circ T_{\alpha, \beta} \circ \cdots \circ T_{\alpha, \beta}}_{n \text{ times}}.
$$

\begin{corollary} \label{002}
If $f : \mathbb{R} \rightarrow \mathbb{R}$ is integrable with respect to $\mu_{\alpha, \beta}$, then
\begin{equation}\label{eq: birk}
\lim_{N \rightarrow \infty} { 1 \over N } \sum_{n = 0}^{N - 1} f \circ T_{\alpha, \beta}^{n} x = {\alpha \over \pi} \int_\mathbb{R} { f(\tau) d \tau \over \alpha^2 + (\tau - \beta)^2 }
\end{equation}
for almost all $x \in \mathbb{R}$.
\end{corollary}

See \cite[Theorem $2.30$]{EW11} for the proof of Birkhoff's ergodic theorem. 
Corollary \ref{002} follows immediately from Birkhoff's ergodic theorem and Theorem \ref{001}.

Birkhoff's ergodic theorem describes the relation between the space average of a function and the time average along the orbit.
In the next section, we will apply Corollary \ref{002} to transform a mean-value of ergodic type into a computable integral.

In the rest of this section, we complete the proof of Theorem \ref{001}.
We first recall the famous result given by R. Adler and B. Weiss.

\begin{lemma} \label{003}
The Boolean transformation $T_{1, 0}$ is measure preserving with respect to $\nu$.
Moreover, it is ergodic.
\end{lemma}

See \cite[Theorem and Main Theorem]{AW73} for the proof of Lemma \ref{003}.

%%%
\begin{proof}[Proof of Theorem \ref{001}]
We first check that $T: = T_{1, 0}$ is measure preserving and ergodic with respect to $\mu := \mu_{1, 0}$.
Let us denote by $\chi_A$ the {\it indicator function} of $A \subset \mathbb{R}$.
It follows from a simple calculation that
\begin{align*}
\mu(T^{-1}(A)) 
&= {1 \over \pi} \int_{\mathbb{R}} { \chi_A(T(\tau)) d \tau \over 1 + \tau^2 } \\
%&= {1 \over \pi} \left( \int_{-\infty}^{0} { \chi_A(T(\tau)) d \tau \over 1 + \tau^2 } + \int_{0}^{\infty} { \chi_A(T(\tau)) d \tau \over 1 + \tau^2 } \right) \\
%&= {1 \over \pi} \left( {1 \over 2 } \int_{\mathbb{R}} { \chi_A(T(\tau)) d T(\tau) \over 1 + T(\tau)^2 } + {1 \over 2 } \int_{\mathbb{R}} { \chi_A(T(\tau)) d T(\tau) \over 1 + T(\tau)^2 } \right) \\
&= {1 \over \pi} \int_{\mathbb{R}} { \chi_A(T(\tau)) d T(\tau) \over 1 + T(\tau)^2 }
= \mu(A)
\end{align*}
for any $A \in \mathcal{B}$.
%Note that each restriction $T | (- \infty, 0), (0, \infty) \rightarrow \mathbb{R}$ of $T$ to $(-\infty, 0)$ and $(0, \infty)$ is bijective.
%See also \cite[Lemma $2.6$ and Example $2.10$]{EW11} for more details.
Thus $T$ is measure preserving with respect to $\mu$.
If $T^{-1}(A) = A$, it follows from Lemma \ref{003} that either $\nu(A)$ or $\nu(X \backslash A)$ is $0$.
Hence, by \eqref{eq: zero}, either $\mu(A)$ or $\mu(X \backslash A)$ must be $0$.

Next, let us consider the general case.
Defining the affine transformation $\phi_{\alpha, \beta} : \mathbb{R} \rightarrow \mathbb{R}$ by $\phi_{\alpha, \beta} (x) := \alpha x + \beta$,
we can easily check that
$$%\begin{equation}\label{eq: eq1}
T_{\alpha, \beta} = \phi_{\alpha, \beta} \circ T \circ \phi^{-1}_{\alpha, \beta}
$$%\end{equation}
and
$$%\begin{equation}\label{eq: eq2}
\mu_{\alpha, \beta}(A) = \mu(\phi_{\alpha, \beta}^{-1}(A)).
$$%\end{equation}
Since $T$ is measure preserving with respect to $\mu$, we have
\begin{align*}
\mu_{\alpha, \beta}(T_{\alpha, \beta}^{-1}(A))
&= \mu(\phi_{\alpha, \beta}^{-1}(T^{-1}_{\alpha, \beta}(A)) \\
&= \mu(\phi_{\alpha, \beta}^{-1}(\phi_{\alpha, \beta} ( T^{-1} ( \phi^{-1}_{\alpha, \beta}(A)))) \\
&= \mu(T^{-1}(\phi_{\alpha, \beta}^{-1}(A) ) ) \\
&= \mu(\phi_{\alpha, \beta}^{-1}(A)) = \mu_{\alpha, \beta}(A).
\end{align*}
Moreover, if $T^{-1}_{\alpha, \beta}(A) = A$, we have
$$
T^{-1}(\phi_{\alpha, \beta}^{-1}(A) ) = \phi^{-1}_{\alpha, \beta} ( T_{\alpha, \beta}^{-1} (A) ) = \phi_{\alpha, \beta}^{-1}(A).
$$
Since $T$ is ergodic with respect to $\mu$,  either $\mu_{\alpha, \beta}(A) = \mu(\phi_{\alpha, \beta}^{-1}(A))$ or $\mu_{\alpha, \beta}(X \backslash A) = \mu(X \backslash \phi_{\alpha, \beta}^{-1}(A))$ is $0$.
\end{proof}

%%%%%%%%%%%%%%%%LEE%%%%%%%%%%%%%%%%%%%%%%
%%%%%%%%%%%%%%%%LEE%%%%%%%%%%%%%%%%%%%%%%

%%%%%%%%%%%%%%%%SUR%%%%%%%%%%%%%%%%%%%%%%
%%%%%%%%%%%%%%%%SUR%%%%%%%%%%%%%%%%%%%%%%

\section{Proof of the main theorem} \label{sec_proof}

\begin{proof}[Proof of Theorem \ref{birk}]
It follows from Corollary \ref{002} that \eqref{eq:birk} holds. 
For the case $m=1$, we set the values of the integrand to be the principal value on the line $\mathbb{L}_{\sigma_0}$ and since this is integrable, as we shall see below in Case 3 of the evaluation of $l_{\alpha,\beta}$, we can now apply Corollary \ref{002} for all $s \in \mathbb{H}_{c}$.
In the rest of this section, we evaluate $l_{\alpha,\beta}$ to complete the proof of Theorem \ref{birk}.

%
%\begin{enumerate}
%
%
%\item[(a)] 

Suppose that $f$ has no pole in $\mathbb{H}_{c}$.
The poles of the integrand in $l_{\alpha,\beta}$ in $\mathbb{H}_{c}$ are coming only from the zeros of $ {\alpha^2+(\tau-\beta)^2}.$
%We evaluate $l_{\alpha,\beta}$ as follows:
For any $s=\sigma+it \in \mathbb{H}_{c}$, we consider the counterclockwise oriented semicircle $\Gamma_R$ of a sufficiently large radius $R>|s|+\alpha+|\beta|$ centered at the origin. %as in Figure \ref{case0}.
%We let $R>|s|+|s_0|+\alpha+|\beta|$ and denote by $\Gamma_R$ the counterclockwise oriented semicircle $\{ R \exp(\theta i) \mid \pi \leq \theta \leq 2\pi \}$ (see Figure \ref{case0}).
%
Then applying Cauchy's integral theorem, we have
\begin{align*}
l_{\alpha,\beta}(s)
&= \frac{\alpha}{\pi} \int_{\mathbb{R}} \frac{f(s+i\tau)}{\alpha^2 +(\tau-\beta)^2} d\tau \\
&= \frac{\alpha}{\pi} \left( \lim_{R\rightarrow\infty} \int_{\Gamma_R} \frac{f(s+i\tau)}{\alpha^2+(\tau-\beta)^2} d\tau
- 2\pi i \operatorname{Res}_{\tau=\beta-i\alpha} \frac{f(s+i\tau)}{\alpha^2+(\tau-\beta)^2} \right).
\end{align*}
Note that we can find a $\sigma' \in (c,\sigma)$ sufficiently near $c$.
Setting $\epsilon = (\sigma'-c)/2$, we have
\begin{align*}
\int_{\Gamma_R} \frac{f(s+i\tau)}{\alpha^2+(\tau-\beta)^2} d\tau
&= \int_\pi^{2\pi} \frac{f(s+iRe^{i\theta})}{\alpha^2+(Re^{i\theta}-\beta)^2} iRe^{i\theta}d\theta \\
%&\ll_{\alpha,\beta} \frac{1}{R} \int_\pi^{2\pi} |f(s+iRe^{i\theta})| d\theta \\
&\ll_{\alpha,\beta} \frac{1}{R} \left( \int_\pi^{5\pi/4} + \int_{5\pi/4}^{7\pi/4} + \int_{7\pi/4}^{2\pi} \right) |f(s+iRe^{i\theta})| d\theta \\
&\ll \frac{1}{R} \bigg( \max_{\theta\in[\pi,5\pi/4]\cup[7\pi/4,2\pi]} |f(\sigma-R\sin{\theta}+i(t+R\cos{\theta}))| \\
&\quad\quad\quad\quad
+ \max_{\theta\in[5\pi/4,7\pi/4]} |f(\sigma-R\sin{\theta}+i(t+R\cos{\theta}))| \bigg) \\
&\leq \frac{1}{R} \left( \max_{\theta\in[\pi,5\pi/4]\cup[7\pi/4,2\pi]} |t+R\cos{\theta}|^{\nu(\sigma-R\sin{\theta})+\epsilon} + M \right) \\
%&\leq \frac{1}{R} \left( \max_{\theta\in[\pi,5\pi/4]\cup[7\pi/4,2\pi]} |t+R\cos{\theta}|^{\nu(\sigma)+\epsilon} + M \right) \\
%&\leq \frac{1}{R} \left( \max_{\theta\in[\pi,5\pi/4]\cup[7\pi/4,2\pi]} |t+R\cos{\theta}|^{\nu(\sigma')+\epsilon} + M \right) \\
&\leq \frac{1}{R} \left( \max_{\theta\in[\pi,5\pi/4]\cup[7\pi/4,2\pi]} |t+R\cos{\theta}|^{1+c-\sigma'+\epsilon} + M \right) \\
%&\ll_{f, \epsilon} \frac{1}{R} \left( R^{1+c-\sigma'+\epsilon}\left(\frac{|t|}{R}+1\right)^{1+c-\sigma'+\epsilon} + M \right) \\
&\ll_{f, \epsilon} R^{c-\sigma'+\epsilon}\left(\frac{|t|}{R}+1\right)^{1+c-\sigma'+\epsilon} + \frac{M}{R},
\end{align*}
thus the integral on $\Gamma_R$ vanishes as $R$ tends to $\infty$.
By simple calculations, we find that
\begin{equation} \label{eq:res-i}
\begin{aligned}
\operatorname{Res}_{\tau=\beta-i\alpha} \frac{f(s+i\tau)}{\alpha^2+(\tau-\beta)^2}
&= \lim_{\tau \rightarrow \beta-i\alpha} (\tau-\beta+i\alpha) \times
\frac{f(s+i\tau)}{\alpha^2+(\tau-\beta)^2} \\
&= \frac{f(s+\alpha+i\beta)}{-2\alpha i}.
\end{aligned}
\end{equation}
Hence we obtain \eqref{eq:1} for all $s \in \mathbb{H}_c$. \\

%
%\item[(b)] 

Suppose that $f$ has a pole at $s=s_0=\sigma_0+it_0$ and $\sigma_0>c$.
Now for $s=\sigma+it \in \mathbb{H}_{c}$, the integrand has three simple poles: $\tau=\beta+i\alpha$, $\tau=\beta-i\alpha$, and $\tau = i(s-s_0)$.
Here we divide the proof into three cases according to the condition whether the pole $\tau = i(s-s_0)$ is below ($c<\sigma<\sigma_0$),
%, see Figures \ref{case11} and \ref{case12}),
above ($\sigma>\sigma_0$),
%see Figure \ref{case2}),
or on the real line ($\sigma=\sigma_0$) of the $\tau$-plane. \\
%, see Figure \ref{case3})
%Here we divide the proof into three cases according to the condition whether the pole $\tau = i(s-s_0)$ is below ($s \in \mathbb{H}_{c} \backslash (\mathbb{L}_{\sigma_0}\cup\mathbb{H}_{\sigma_0})$), above ($s \in \mathbb{H}_{\sigma_0}$), or on the real line ($s \in \mathbb{L}_{\sigma_0}$) of the $\tau$-plane. \\

%
\noindent
\underline{\textbf{Case 1:}} $\operatorname{Im}(i(s-s_0))<0$. \vskip1mm
We first consider when $i(s-s_0) \neq \beta-i\alpha$ and let $\Gamma_R$ be the counterclockwise oriented semicircle of a large radius $R>|s|+|s_0|+\alpha+|\beta|$ centered at the origin. %as in Figure \ref{case11}.

\noindent
Again by applying Cauchy's integral theorem, we can show that
\begin{align*}
l_{\alpha,\beta}(s)
&= \frac{\alpha}{\pi} \int_{\mathbb{R}} \frac{f(s+i\tau)}{\alpha^2 +(\tau-\beta)^2} d\tau \\
&= - 2\alpha i \left( \operatorname{Res}_{\tau=\beta-i\alpha} \frac{f(s+i\tau)}{\alpha^2+(\tau-\beta)^2}
+ \operatorname{Res}_{\tau=i(s-s_0)} \frac{f(s+i\tau)}{\alpha^2+(\tau-\beta)^2} \right).
\end{align*}
Substituting \eqref{eq:res-i} into the above, we obtain
$$
l_{\alpha,\beta}(s)
= f(s+\alpha+i\beta) - 2\alpha i \operatorname{Res}_{\tau=i(s-s_0)} \frac{f(s+i\tau)}{\alpha^2+(\tau-\beta)^2}.
$$
From the Laurent expansion \eqref{eq:laurent} of $f$, we can calculate
\begin{align*}
\frac{f(s+i\tau)}{\alpha^2+(\tau-\beta)^2}
&= \frac{i}{2\alpha} \left(  \sum_{n=0}^\infty \frac{(\tau-i(s-s_0))^n}{(\beta+i\alpha-i(s-s_0))^{n+1}}
- \sum_{n=0}^\infty \frac{(\tau-i(s-s_0))^n}{(\beta-i\alpha-i(s-s_0))^{n+1}} \right) \\
&\quad\quad\times \sum_{n=-m}^\infty a_n i^n (\tau-i(s-s_0))^n.
\end{align*}
Thus
%\begin{align*}
%- 2\alpha i \operatorname{Res}_{\tau=i(s-s_0)} &\frac{f(s+i\tau)}{\alpha^2+(\tau-\beta)^2} \\
%&= \frac{1}{-\beta+i\alpha+i(s-s_0)} \sum_{n=-m}^{-1} a_n i^n \left( \beta-i\alpha-i(s-s_0) \right)^{n+1} \\
%&\quad\quad+ \frac{1}{\beta+i\alpha-i(s-s_0)} \sum_{n=-m}^{-1} a_n i^n \left( \beta+i\alpha-i(s-s_0) \right)^{n+1}.
%\end{align*}
\begin{align*}
- 2\alpha i \operatorname{Res}_{\tau=i(s-s_0)} &\frac{f(s+i\tau)}{\alpha^2+(\tau-\beta)^2} \\
&= \sum_{n=1}^m \frac{a_{-n}}{i^n \left( \beta+i\alpha-i(s-s_0) \right)^n}
- \sum_{n=1}^m \frac{a_{-n}}{i^n \left( \beta-i\alpha-i(s-s_0) \right)^n}.
\end{align*}
Thus combining the above calculations and setting
$$
B_m(s_0) := \sum_{n=1}^m \frac{a_{-n}}{i^n \left( \beta+i\alpha-i(s-s_0) \right)^n}
- \sum_{n=1}^m \frac{a_{-n}}{i^n \left( \beta-i\alpha-i(s-s_0) \right)^n},
$$
we obtain
\begin{align*}
l_{\alpha,\beta}(s)
= \frac{\alpha}{\pi} \int_{\mathbb{R}} \frac{f(s+i\tau)}{\alpha^2+(\tau-\beta)^2} d\tau
= f(s+\alpha+i\beta) + B_m(s_0).
\end{align*}
This is the first equation of \eqref{eq:2}.

Now suppose that $i(s-s_0) = \beta-i\alpha$. %(see Figure \ref{case12}).
This case only appears when $\sigma_0-\alpha>c$.
By calculations similar to the above, we have
$$
l_{\alpha,\beta}(s)
= \frac{\alpha}{\pi} \int_{\mathbb{R}} \frac{f(s+i\tau)}{\alpha^2+(\tau-\beta)^2} d\tau
= - 2\alpha i \operatorname{Res}_{\substack{\tau=\beta-i\alpha\\ = i(s-s_0)}} \frac{f(s+i\tau)}{\alpha^2+(\tau-\beta)^2}.
$$
We consider the Laurent expansion of the integrand near $\tau=\beta-i\alpha=i(s-s_0)$:
\begin{align*}
\frac{f(s+i\tau)}{\alpha^2+(\tau-\beta)^2}
&= \left( \sum_{n=-m}^\infty a_n i^n (\tau-\beta+i\alpha)^n \right) \\
&\quad\quad\quad\quad\quad\quad\quad\quad\times 
\frac{1}{\tau-\beta+i\alpha} \times \frac{1}{-2\alpha i} \times \frac{1}{1 - \frac{\tau-\beta+i\alpha}{2\alpha i}} \\
&= \frac{1}{-2\alpha i} \times \frac{1}{\tau-\beta+i\alpha} \times \frac{1}{1 - \frac{\tau-\beta+i\alpha}{2\alpha i}} \sum_{n=-m}^\infty a_n i^n (\tau-\beta+i\alpha)^n \\
&= \frac{1}{-2\alpha i} \frac{1}{\tau-\beta+i\alpha}
\sum_{n=0}^\infty \left(\frac{\tau-\beta+i\alpha}{2\alpha i}\right)^n
\sum_{n=-m}^\infty a_n i^n (\tau-\beta+i\alpha)^n.
\end{align*}
Hence,
$$
\operatorname{Res}_{\substack{\tau=\beta-i\alpha\\ = i(s-s_0)}} \frac{f(s+i\tau)}{\alpha^2+(\tau-\beta)^2}
= \frac{1}{-2\alpha i} \sum_{n=0}^m \frac{a_{-n}}{(-2\alpha)^n}
$$
and so we obtain the second equation of \eqref{eq:2}. \\

\noindent
\underline{\textbf{Case 2:}} $\operatorname{Im}(i(s-s_0)) > 0$. \vskip1mm

In this case, the integrand of $l_{\alpha,\beta}(s)$ has only one pole in the lower half-plane. %(see Figure \ref{case2}).
Thus by a method similar to the case when $f$ has no pole in $\mathbb{H}_c$, we can show that
\begin{align*}
\int_{\mathbb{R}} \frac{f(s+i\tau)}{\alpha^2+(\tau-\beta)^2} d\tau
&= - 2\pi i \operatorname{Res}_{\tau=\beta-i\alpha} \frac{f(s+i\tau)}{\alpha^2+(\tau-\beta)^2}
\overset{\eqref{eq:res-i}}= - 2\pi i \times \frac{f(s+\alpha+i\beta)}{-2\alpha i} \\
&= \frac{\pi}{\alpha} f(s+\alpha+i\beta).
\end{align*}
Thus
$$
l_{\alpha,\beta}(s)
= \frac{\alpha}{\pi} \int_{\mathbb{R}} \frac{f(s+i\tau)}{\alpha^2+(\tau-\beta)^2} d\tau
= f(s+\alpha+i\beta).
$$
This is the third equation of \eqref{eq:2}. \\

\noindent
\underline{\textbf{Case 3:}} $\operatorname{Im}(i(s-s_0)) = 0$ (only for the case $m=1$).
\vskip1mm

Since $\operatorname{Im}(i(s-s_0)) = 0$ ($s_0=\sigma_0+it_0$), $s$ satisfies $\operatorname{Re}(s)=\sigma_0$ in this case.
For convenience, we write $s=\sigma_0+it$.
In this case, we take the principal value of the integrand as in \cite[p. 367]{ste} and so we obtain
\begin{equation} \label{eq:case3}
\begin{aligned}
\int_{\mathbb{R}} \frac{f(\sigma_0+i(t+\tau))}{\alpha^2+(\tau-\beta)^2} d\tau
= \lim_{\substack{R\rightarrow\infty\\ \epsilon\rightarrow0^+}}
\bigg( \int_{C_R} &- \int_{C_\epsilon} \bigg) \frac{f(\sigma_0+i(t+\tau))}{\alpha^2+(\tau-\beta)^2} d\tau \\
&- 2\pi i \operatorname{Res}_{\tau=\beta-i\alpha} \frac{f(\sigma_0+i(t+\tau))}{\alpha^2+(\tau-\beta)^2},
\end{aligned}
\end{equation}
where $C_R$ and $C_\epsilon$ are the counterclockwise oriented semicircles of radius $R$ ($R>1+|s|+\alpha+|\beta|$) and $\epsilon$ centered at $\tau = t_0-t$ located in the lower half of the $\tau$-plane. %(see Figure \ref{case3}).

As the other cases, the integral along $C_R$ vanishes as $R$ tends to $\infty$.
On the other hand, the integral along $C_\epsilon$ is evaluated as
\begin{align*}
\int_{C_\epsilon} \frac{f(\sigma_0+i(t+\tau))}{\alpha^2+(\tau-\beta)^2} d\tau
&= \int_\pi^{2\pi} \frac{f(\sigma_0+i(t_0+\epsilon e^{i\theta}))}{\alpha^2 + (t_0 - t + \epsilon e^{i\theta} - \beta)^2} i\epsilon e^{i\theta} d\theta \\
&= \int_\pi^{2\pi} \left( \frac{a_{-1}}{i\epsilon e^{i\theta}} + O(1) \right) \frac{i\epsilon e^{i\theta}}{\alpha^2 + (t_0 - t + \epsilon e^{i\theta} - \beta)^2} d\theta
\end{align*}
hence 
$$
\lim_{\epsilon \rightarrow 0^+} \int_{C_\epsilon} \frac{f(\sigma_0+i(t+\tau))}{\alpha^2+(\tau-\beta)^2} d\tau = \frac{a_{-1}\pi}{\alpha^2+(t_0-t-\beta)^2}.
$$
Again from \eqref{eq:res-i},
$$
\operatorname{Res}_{\tau=\beta-i\alpha} \frac{f(\sigma_0+i(t+\tau))}{\alpha^2+(\tau-\beta)^2}
= \frac{f(\sigma_0+\alpha+i(t+\beta))}{-2\alpha i}.
$$
These imply that \eqref{eq:3} holds.
%
%
%\end{enumerate}
%
\end{proof}

Remark that the method used in Case $3$ in the proof does not work if $m>1$.

%%%%%%%%%%%%%%%applications%%%%%%%%%%%%%%
%%%%%%%%%%%%%%%applications%%%%%%%%%%%%%%

%\section{Applications of Theorem \ref{birk} to zeta functions and $L$-functions} \label{sec_exa}

%Before we begin, we again note that Theorem \ref{birk} is applicable to not only zeta functions and $L$-functions in the Selberg class but also to some other zeta functions and $L$-functions.
%For example, as we shall see in the following Example \ref{birk2}, Dirichlet $L$-functions associated with non-primitive characters also satisfy the conditions in Theorem \ref{birk} and we can use Theorem \ref{birk} to calculate their mean values of type \eqref{eq:birk}.
%Another example is given by Hurwitz zeta functions which will not be explained further in this paper.

%We recall that for a given positive real number $\alpha$ and a real number $\beta$, the affine Boolean transformation $T_{\alpha,\beta}$ is as given in \eqref{eq:affboot}.
%From Theorem \ref{birk}, we can easily obtain the following result for zeta functions and $L$-functions.

%%%%%%%%%%%%%%%%SUR%%%%%%%%%%%%%%%%%%%%%%
%%%%%%%%%%%%%%%%SUR%%%%%%%%%%%%%%%%%%%%%%

\section*{Acknowledgements}

The authors would like to express their gratitude to Prof. Kohji Matsumoto, Prof. J\"{o}rn Steuding, Prof. Takashi Nakamura, and Dr. \L ukasz Pa\'{n}kowski for their valuable advice.
%They also would like to deeply thank the referees for their efforts.


\bibliographystyle{amsalpha}

\end{document}